\documentclass[12pt]{article}
\usepackage{amsmath,amssymb,amsbsy,amsfonts,amsthm}
\usepackage{url}

\begin{document}

\newtheorem{theorem}{Theorem}
\newtheorem{lemma}[theorem]{Lemma}
\newtheorem{algol}{Algorithm}
\newtheorem{cor}[theorem]{Corollary}
\newtheorem{prop}[theorem]{Proposition}

\newcommand{\comm}[1]{\marginpar{%
\vskip-\baselineskip 
\raggedright\footnotesize
\itshape\hrule\smallskip#1\par\smallskip\hrule}}

\def\cA{{\mathcal A}}
\def\cB{{\mathcal B}}
\def\cC{{\mathcal C}}
\def\cD{{\mathcal D}}
\def\cE{{\mathcal E}}
\def\cF{{\mathcal F}}
\def\cG{{\mathcal G}}
\def\cH{{\mathcal H}}
\def\cI{{\mathcal I}}
\def\cJ{{\mathcal J}}
\def\cK{{\mathcal K}}
\def\cL{{\mathcal L}}
\def\cM{{\mathcal M}}
\def\cN{{\mathcal N}}
\def\cO{{\mathcal O}}
\def\cP{{\mathcal P}}
\def\cQ{{\mathcal Q}}
\def\cR{{\mathcal R}}
\def\cS{{\mathcal S}}
\def\cT{{\mathcal T}}
\def\cU{{\mathcal U}}
\def\cV{{\mathcal V}}
\def\cW{{\mathcal W}}
\def\cX{{\mathcal X}}
\def\cY{{\mathcal Y}}
\def\cZ{{\mathcal Z}}

\def\C{\mathbb{C}}
\def\F{\mathbb{F}}
\def\K{\mathbb{K}}
\def\Z{\mathbb{Z}}
\def\R{\mathbb{R}}
\def\Q{\mathbb{Q}}
\def\N{\mathbb{N}}
\def\M{\textsf{M}}

\def\({\left(}
\def\){\right)}
\def\[{\left[}
\def\]{\right]}
\def\<{\langle}
\def\>{\rangle}

\def\e{e}

\def\eq{\e_q}
\def\fS{{\mathfrak S}}

\def\lcm{{\mathrm{lcm}}\,}

\def\fl#1{\left\lfloor#1\right\rfloor}
\def\rf#1{\left\lceil#1\right\rceil}
\def\mand{\qquad\mbox{and}\qquad}

\def\jt{\tilde\jmath}
\def\ellmax{\ell_{\rm max}}
\def\llog{\log\log}


\title{On the Product of Small Elkies Primes}

\author{
{\sc Igor E.~Shparlinski} \\
{Department of Computing}\\
{Macquarie University} \\
{Sydney, NSW 2109, Australia} \\
{\tt igor.shparlinski@mq.edu.au}
}

\maketitle

\begin{abstract} Given an elliptic curve $E$
over a finite field $\F_q$ of $q$ elements, we say that an odd prime 
$\ell \nmid q$ is an Elkies prime for $E$ if $t_E^2 - 4q$ is a quadratic residue 
modulo~$\ell$, where $t_E = q+1 - \#E(\F_q)$  and $\#E(\F_q)$ is the number of $\F_q$-rational
points on $E$. 
These primes are used in  the presently most efficient algorithm to compute $\#E(\F_q)$. 
In particular, the bound $L_q(E)$ such that the
product of all Elkies primes for $E$  up to $L_q(E)$  exceeds $4q^{1/2}$ is  
a crucial parameter of this algorithm.
We show that there are infinitely many pairs $(p, E)$ of 
primes $p$ and curves $E$ over $\F_p$ with   $L_p(E)  \ge c \log p \log \log \log p$
for some absolute constant $c>0$, while a naive heuristic estimate suggests that 
 $L_p(E) \sim \log p$. 
This complements recent 
results of Galbraith and Satoh (2002), conditional 
under the  Generalised Riemann Hypothesis, and of 
Shparlinski and Sutherland (2012), unconditional for 
almost all pairs $(p,E)$. 
\end{abstract}


\section{Introduction}

For an elliptic curve $E$ over  a finite field $\F_q$ 
of $q$ elements we denote by $\#E(\F_q)$  the number of
$\F_q$-rational points on $E$ and define the {\it trace of Frobenius\/}
$t_E = q+1 - \#E(\F_q)$; we refer to~\cite{ACDFLNV,Silv} for a background on 
elliptic curves.
We say that an odd prime 
$\ell \nmid q$ is an  {\it Elkies prime\/} for $E$ if $t_E^2 - 4q$ is a quadratic residue
modulo $\ell$; otherwise $\ell\nmid q$ is called an {\it  Atkin prime\/}.

These primes play a key role in the {\it Schoof-Elkies-Atkin (SEA) algorithm\/}, 
see~\cite[Sections~17.2.2 and~17.2.5]{ACDFLNV}, and their distribution affects the performance of this
algorithm in a rather dramatic way. Thus, for 
an elliptic curve $E$ over $\F_q$, we define
$N_a(E;L)$ and $N_e(E;L)$ as the numbers of Atkin and Elkies primes 
$\ell  \in [1,L]$, respectively. 
Obviously,
$$
N_a(E;L) + N_e(E;L) =  \pi(L) + O\(1\),
$$
where $\pi(L)$ denotes the number of primes $\ell < L$.
Furthermore, for any elliptic curve over a finite field, one expects about the same number of  
Atkin and Elkies primes $\ell < L$ as $L\to \infty$. That is,  naive heuristic suggests 
that 
\begin{equation}
\label{eq:NaNe}
N_a(E;L) \sim N_e(E;L)  \sim   \frac12  \pi(L),
\end{equation}
as $L\to \infty$. 

It has been noted by Galbraith and Satoh~\cite[Appendix~A]{Sat}, that under the Generalised Riemann Hypothesis (GRH), using the bound 
on sums of quadratic characters over primes, 
one derives that~\eqref{eq:NaNe}
holds for $L \ge (\log q)^{2 + \varepsilon}$ for any fixed $\varepsilon>0$
and a sufficiently large $q$.   

The unconditional results are much weaker 
and essentially rely on our knowledge of the distribution of 
primes in arithmetic progressions; see~\cite[Section~5.9]{IwKow}
or~\cite[Chapters~4 and~11]{MonVau}. However, for almost all pairs
$(p, E)$ of  primes $p$ and elliptic curves $E$ over $\F_p$, 
Shparlinski and Sutherland~\cite{ShpSuth} have 
established the asymtotic formula~\eqref{eq:NaNe} for $L \ge (\log p)^{\varepsilon}$ 
for any fixed $\varepsilon>0$, that is, starting from much smaller 
values of $L$ that those implied by the GRH. 
In particular, Let $\cL_E(p)$ be the set all   Elkies primes for  an elliptic curve $E$
over $\F_p$.  
We see that the prime number theorem
and the result of~\cite{ShpSuth}  implies 
that for some function $L(p) \sim \log p$
for almost all pairs $(p, E)$ we have
\begin{equation}
\label{eq:Prod}
\prod_{\substack{\ell \in \cL_E(p)\\ 3 \le \ell \le L(p)}} \ell > 4p^{1/2}.
\end{equation}
 Note that this condition is crucial for the SEA
point counting algorithm,  see~\cite[Sections~17.2.2 and~17.2.5]{ACDFLNV}.

Here we show that this ``almost all'' result cannot be extended for 
all primes and curves even for a slightly larger values of $L(p)$. 
More precisely, we show that there is an absolute constant $c>0$ such 
that for any function 
$L(p) \le c \log p \log \log \log p$
the inequality~\eqref{eq:Prod} fails in a very strong sense for infinitely many pairs  $(p, E)$.

\begin{theorem} 
\label{thm:Low Elkies} There is a constant $c>0$ so that for infinitely 
many pairs $(p, E)$ of  primes $p$ and curves $E$ over $\F_p$, and $L \le c \log p \log \log \log p$ we have
$$
\prod_{\substack{\ell \in \cL_E(p)\\ 3 \le \ell \le L}} \ell = p^{o(1)}. 
$$
\end{theorem} 

We note that Galbraith and Satoh~\cite[Appendix~A]{Sat} have conjectured  
and actually presented some arguments supporting a result of this kind. 
Moreover, under both the GRH  and 
the conjecture that every positive integer $n \equiv 1 \pmod 4$ can be represented as $n = 4p-t^2$ the argument of Galbraith and Satoh~\cite[Appendix~A]{Sat}
can be made rigorous and in fact under these assumptions it allows to replace 
$\log p \log \log \log p$ with $\log p  \log \log p$ in Theorem~\ref{thm:Low Elkies}.
Unfortunately, presently the required representation $n = 4p-t^2$  is known to
exist only for almost all $n$ (see~\cite{BaZh,LuSu}), 
which is not enough to complete the argument (even under the GRH).

\section{Preparations}
\label{sec:prep}

We recall
the notations $U = O(V)$, $V = \Omega(U)$, $U \ll V$ and  $V \gg U$, which are all
equivalent to the statement that the inequality $|U| \le c\,V$ holds asymptotically,
with some constant $c> 0$.  

We always assume  that $\ell$ and $p$ run through the prime values. 

For integers $a$ and $m \ge 2$, 
we use $(a/m)$ to denote a Jacobi symbol of $a$ modulo $m$, see~\cite[Section~3.5]{IwKow}.
We also use $\tau(k)$ and $\mu(k)$  to denote 
the number of integer positive divisors and the M{\"o}bius function of $k\ge 1$.
It is easy to see that for a square-free $k$ we have 
$$
\tau(k) = 2^{\omega(k)}
$$
where $\omega(k)$ is the number of prime divisors of $k$.  

Our main tools are bounds of multiplicative character sums.

The following estimate is a slight 
generalisation of~\cite[Lemma~2.2]{LuSh} and is also given in~\cite{ShpSuth}.

\begin{lemma} 
\label{lem:Long}  For any integers $a$ and $T\ge 1$  and a product 
$m = \ell_1 \ldots \ell_s$ 
of $s\ge 0$ distinct odd primes $\ell_1, \ldots,\ell_s$ 
with $\gcd(a,m)=1$ we have
$$
\sum_{|t|\le T} \(\frac{t^2 - a}{m}\) \ll T/m + C^s m^{1/2} \log m,
$$
for some absolute constant $C \ge 1$.
\end{lemma}

We also need a slight extension of~\cite[Corollary~12.14]{IwKow}.
In fact, we present it in much wider generality and strength than 
is needed for our purpose. 
First we note that for a square-free integer $m$ and any integers 
$u$ and $v$, we have 
\begin{equation}
\label{eq:gcd}
\gcd((u-v)^2,m)= \gcd(u-v,m).
\end{equation}
Hence, in the case of quadratic polynomials, the bound of~\cite[Theorem~12.10]{IwKow}, implies 
the following results"

\begin{lemma} 
\label{lem:Short}  Assume that a square-free odd integer  $m\ge 3$ and 
an arbitrary integer $N\ge 1$ are such that 
all prime factors of $m$  are at most $N^{1/9}$. Then for 
any two integers $u, v$
we have 
$$
 \left| \sum_{n=1}^N \(\frac{(n-u)(n-v)}{m}\)\right| 
\le 4 N \(\gcd(u-v,m)   m^{-1}\tau(m)^{r^2+2r}\)^{1/r2^r} , 
$$
where $r$ is any positive integer with $N^r > m^3$.
\end{lemma} 

\begin{proof} As in the proof of~\cite[Corollary~12.14]{IwKow}, 
we note that there is a factorisation
$$
m= m_1\ldots m_r
$$
with $m_j \le N^{4/9}$, $j =1, \ldots, r$.
In particular, by~\cite[Theorem~12.10]{IwKow}, recalling~\eqref{eq:gcd}, we see that for any 
$j =1, \ldots, r$ we have
$$
 \left| \sum_{n=1}^N \(\frac{(n-u)(n-v)}{m}\)\right| 
\le  4 N \(\gcd(u-v,m_j)  m_j^{-1}\tau(m_j)^{r^2+2r}\)^{1/2^r}.
$$
Since $m$ is square-free, we see that 
$m_1,\ldots, m_r$ are relatively prime.
Using the multiplicativity the divisor function, we 
obtain
$$
\prod_{j=1}^r \gcd(u-v,m_j)  m_j^{-1}\tau(m_j)^{r^2+2r} 
= \gcd(u-v,m)  m^{-1}\tau(m)^{r^2+2r}. 
$$
Therefore, for some    $j \in \{1, \ldots, r\}$ we have
$$
\gcd(u-v,m_j)  m_j^{-1}\tau(m_j)^{r^2+2r} 
\le  \(\gcd(u-v,m)  m^{-1}\tau(m)^{r^2+2r}\)^{1/r}
$$ 
and the result now follows. 
\end{proof}

We remark that several more stronger and more general results of this type have 
recently been given by Chang~\cite{Chang}.

Furthermore, we also recall the  following classical
results of Deuring~\cite{Deur}.

\begin{lemma} 
\label{lem:t p} For any prime $p$ and an integer $t$ with $|t| \le 2q^{1/2}$, there is  a curve 
$E$ over $\F_p$ with $\#E(\F_p) = p + 1 -t$.
\end{lemma}

\section{Proof of Theorem~\ref{thm:Low Elkies}}

Let $Q$ be a sufficiently large integer. We then set
$$
L  = \fl{0.3 \log Q\log \log \log Q}, \quad M  = \fl{\log Q\(\log \log \log Q\)^{-1}}, \quad 
T = \fl{Q^{1/2}}.
$$

Since, by the prime number theorem
$$
\prod_{\ell \in \le M} \ell = Q^{o(1)},
$$
we see from Lemma~\ref{lem:t p} that it is enough to 
show that  for 
any sufficiently large $Q$,  there is an integer 
$t \in [1, T]$ and a prime $p \in [Q/2, Q]$ such that 
\begin{equation}
\label{eq:Cond 1}
 \(\frac{t^2 - 4p}{\ell}\) \ne 1
\end{equation}
for all primes  $\ell \in  [M,L]$.

Clearly, if the condition~\eqref{eq:Cond 1} is violated, then 
$$
\prod_{\ell \in  [M,L]}  \(1 - \(\frac{t^2 - 4p}{\ell}\) \) = 0.
$$
Thus it is enough to show that the sum
$$
W = \sum_{1 \le t\le T} \, \sum_{Q/2 \le p \le Q}\, \prod_{\ell \in  [M,L]} \(1 + \(\frac{t^2 - 4p}{\ell}\) \) 
$$
is positive, that is, that
\begin{equation}
\label{eq:W pos}
W > 0
\end{equation}
for the above choice of $L$, $M$ and $T$, 
provided that $Q$ is  sufficiently large.

Let $\cM$ be the set of $2^{\pi(L) - \pi(M)}$ square-free products (including the empty product) 
composed out of primes $\ell \in [M,L]$,
and let $\cM^* = \cM\setminus\{1\}$.
We have
$$
W   =    \sum_{1 \le t\le T}\, \sum_{Q/2 \le p \le Q} \mu(m)  \sum_{m\in \cM} \(\frac{t^2 - 4p}{m}\) . 
$$
Changing the order of summation and separating the term $T (\pi(Q)-\pi(Q/2))$
corresponding to $m = 1$, we derive
\begin{equation}
\label{eq:W Sm}
W = T (\pi(Q)-\pi(Q/2)) + \sum_{m \in \cM^*}  \mu(m) S(m)
\end{equation}
where 
$$
S(m) =  \sum_{1 \le t\le T} \,\sum_{Q/2 \le p \le Q}  \(\frac{t^2 - 4p}{m}\) .
$$
Thus
$$
|S(m)| \le  \sum_{Q/2 \le p \le Q} \left| \sum_{1 \le t\le T}  \(\frac{t^2 - 4p}{m}\)\right|.
$$

For $m \le T^{1/4}$ we use Lemma~\ref{lem:Long} and note that 
$$
C^{\omega(m)} = \tau(m)^{\log C/\log 2}= m^{o(1)},
$$ 
so we obtain 
$$
S(m)  \ll \pi(Q)\(T/m + C^s m^{1/2} \log m\) \ll \pi(Q) T/m.
$$
Thus for the contribution from all such sums we derive
\begin{equation}
\label{eq:Small m}
\sum_{\substack{m\in \cM^*\\m \le T^{1/4}}}  |S(m)| \ll 
\pi(Q) T \sum_{\substack{m\in \cM^*\\m \le T^{1/4}}} 1/m
\ll \pi(Q) T \( \prod_{\ell \in  [M,L]} \(1 + \frac{1}{\ell}\) - 1\).
\end{equation}
Furthermore
$$
\log \prod_{\ell \in  [M,L]} \(1 + \frac{1}{\ell}\)   = 
\sum_{\ell \in  [M,L]} \log \(1 + \frac{1}{\ell}\) 
\ll \sum_{\ell \in  [M,L]}   \frac{1}{\ell} .
$$
By the Mertens theorem,  see~\cite[Equation~(2.15)]{IwKow}, 
\begin{equation*}
\begin{split}
\sum_{\ell \in  [M,L]}   \frac{1}{\ell}
&= \log \frac{\log  L}{\log  M} + O(1/\log M) \\
&=  \log \frac{\log \log Q  + \log \log \log \log Q + \log 0.3}
{\log \log Q  - \log \log \log \log Q} + O(1/\log M) \\
&= \log\(1 + O\( \frac{\log \log \log \log Q}{\log \log Q}\)\)  + O(1/\log M)\\
&\ll  \frac{\log \log \log \log Q}{\log \log Q}.
\end{split}
\end{equation*}
Therefore 
$$
 \prod_{\ell \in  [M,L]} \(1 + \frac{1}{\ell}\) 
  = 1 +O\( \frac{\log \log \log \log Q}{\log \log Q}\).
$$ 
Inserting this bound in~\eqref{eq:Small m}, we
obtain 
\begin{equation}
\label{eq:small m}
\begin{split}
\sum_{\substack{m\in \cM^*\\m \le T^{1/4}}}  |S(m)|  \ll \pi(Q) T  \frac{\log \log \log \log Q}{\log \log Q}
= o(\pi(Q) T).
\end{split}
\end{equation}

To estimate the sums $S(m)$ for $m > T^{1/4}$, using the Cauchy inequality and then extending 
the summation range over all integers $n \le 4Q$, we derive
\begin{equation*}
\begin{split}
|S(m)|^2 & =   \pi(Q) \sum_{Q/2 \le p \le Q} \left| \sum_{1 \le t\le T}  \(\frac{t^2 - 4p}{m}\)\right|^2\\
& \le  \pi(Q) \sum_{n \le 4Q} \left| \sum_{1 \le t\le T}  \(\frac{t^2 - n}{m}\)\right|^2\\ 
& =  \pi(Q)    \sum_{1 \le s,t\le T} \, \sum_{n \le 4Q}  \(\frac{(s^2 - n)(t^2 - n)}{m}\) .
\end{split}
\end{equation*}

If $\gcd(s^2-t^2,m) > m^{1/2}$, we estimate the inner sum trivially as $O(Q)$.
The total contribution from such pairs $(s,t)$, is at most 
\begin{equation}
\label{eq:large gcd}
\begin{split}
\sum_{\substack{d\mid m \\ d > m^{1/2}}}
\sum_{\substack{1 \le s,t\le T\\ s^2 \equiv t^2 \pmod d}} 1 &
\le \sum_{\substack{d\mid m \\ d > m^{1/2}}}  T\(T/d+1\) 2^{\omega(d)}\\
& \le  T\(T/m^{1/2}+1\) \tau(m)^2, 
\end{split}
\end{equation}
since for a square-free $d$,  by the Chinese remainder theorem, any 
quadratic congruence of the form $s^2 \equiv a \pmod d$, $1\le s \le d$,
has at most $2^{\omega(d)}$ solutions.

If $\gcd(s^2-t^2,m)\le  m^{1/2}$, we apply Lemma~\ref{lem:Short} to the inner sum, getting
\begin{equation}
\label{eq:small gcd}
\begin{split}
\left| \sum_{n \le 4Q}  \(\frac{(s^2 - n)(t^2 - n)}{m}\)\right| &\le 
16 Q \(\gcd(s^2-t^2,m)   m^{-1}\tau(m)^{r^2+2r}\)^{1/r2^r}\\
& \le  16 Q \(m^{-1/2}\tau(m)^{r^2+2r}\)^{1/r2^r} 
\end{split}
\end{equation}
for any positive integer $r$ with
\begin{equation}
\label{eq:large r}
(4Q)^r > m^3.
\end{equation}

Therefore, combining~\eqref{eq:large gcd} and~\eqref{eq:small gcd}, 
we obtain
\begin{equation}
\begin{split}
\label{eq:prelim}
S(m)^2 \ll  \pi(Q) & QT\(T/m^{1/2}+1\) \tau(m)^2\\
& + \pi(Q) Q T^2 \(m^{-1/2}\tau(m)^{r^2+2r}\)^{1/r2^r} .
\end{split}
\end{equation}

Furthermore, for $m \in \cM$ 
we have
\begin{equation}
\label{eq:tau m}
\tau(m) \le 2^{\pi(L)} = \exp\((\log 2+o(1)) \frac{\log Q \log \log \log Q}{\log \log Q}\) .
\end{equation}
So if 
\begin{equation}
\label{eq:small r}
r^2 +r \le 0.01 \frac{ \log \log  Q}{\log \log \log Q}
\end{equation}
then for $m > T^{1/4}$ we have 
$$
\tau(m)^{r^2+2r}\le Q^{0.01 \log 2 + o(1)} =  T^{0.01 \log 2 + o(1)} \le
m^{0.04 \log 2 + o(1)} \le m^{1/6}, 
$$
provided that $Q$ is large enough.
Hence, 
$$
m^{-1/2}\tau(m)^{r^2+2r} \le m^{-1/3} \le T^{-1/12}.
$$
Furthermore, since~\eqref{eq:tau m} implies that  
$\tau(m)  = T^{o(1)}$ for $m \in \cM$,  
we see that~\eqref{eq:prelim} implies that for  $m > T^{1/4}$, for any 
$r$ satisfying~\eqref{eq:large r} and~\eqref{eq:small r}, we have
$$
S(m) \ll  Q T^{1-1/24r2^r} .
$$
Therefore,
\begin{equation*}
\begin{split}
\sum_{\substack{m\in \cM^*\\ m > T^{1/4}}}  |S(m)| & \ll 2^{\pi(L)} Q T^{1-1/24r2^r}\\
& \le  Q T^{1-1/24r2^r} \exp\((\log 2+o(1)) \frac{\log Q \log \log \log Q}{\log \log Q}\) .
\end{split}
\end{equation*}
In particular, if we set 
$$
r = \fl{ \log \log \log Q}
$$
then 
$$
T^{1/24 r2^r} = \exp\(\frac{\log Q}{(\log \log Q)^{\log 2 + o(1)}}\).
$$
Therefore,
\begin{equation}
\label{eq:large m}
\sum_{\substack{m\in \cM^*\\ m > T^{1/4}}}  |S(m)|  \ll
 Q T^{1-1/25r2^r} =  o(\pi(Q) T).
\end{equation}
It is also obvious that~\eqref{eq:small r} is satisfied for the 
above choice of $r$. 
Furthermore, the 
condition~\eqref{eq:large r} is satisfied as well
because 
$$
(4Q)^r \ge  \exp((1 +o(1)) \log Q \log \log \log Q) 
$$
and
$$
\max_{m \in \cM} m = \exp((1+o(1)) L)
= \exp( (0.3 +o(1)) \log Q \log \log \log Q).
$$

Substituting~\eqref{eq:small m} and~\eqref{eq:large m} in~\eqref{eq:W Sm}, 
we see that~\eqref{eq:W pos} holds, which concludes the proof.


\section*{Acknowledgement}

The author is very grateful to Andrew Sutherland for very useful comments. 

During the preparation of this work the author was supported in part by 
the Australian Research Council grant DP1092835, and Macquarie University 
grant MQRDG1465020.

\end{document}